\documentclass{amsart}

\usepackage{amsmath,amssymb,amsbsy,amsfonts, latexsym, amsopn,amstext, amsxtra,euscript,amscd}

\usepackage[colorlinks=true, urlcolor=blue, pdfborder={0 0 0}]{hyperref}

\usepackage{xcolor}

\usepackage{url}

\usepackage[citation-order]{amsrefs}

\usepackage[all]{xy}

\def\r{\tau}

\def\ch{\mbox{char }}

\def\l{\lambda}

\def\Z{{\mathbb Z}}
\def\C{{\mathbb C}}
\def\P{\mathbb P}

\def\M{{\mathcal M}}
\def\H{\mathcal H}

\def\O{\Omega}

\def\X{\mathcal X}
\def\Y{\mathcal Y}
\def\A{\mathcal A}

\DeclareMathOperator\Aut{Aut }
\DeclareMathOperator\Jac{Jac }
\DeclareMathOperator\bAut{\overline {Aut} }

\DeclareMathOperator\lcm{lcm }

\def\iso{\cong}
\def\embd{\hookrightarrow}

\newcommand\G{\bar G}
\newcommand\bG{\overline G}

\def\D{\Delta}
\def\e{\zeta}

\def\t{\tau}
\def\d{{\delta }}

\def\bC{\mathfrak \si}

\def\<{\langle}
\def\>{\rangle}

\def\r{\gamma}

\def\normal{\triangleleft}

\def\s{\sigma}
\def\bs{\bar \sigma}

\copyrightinfo{2009}{American Mathematical Society}

\newtheorem{theorem}{Theorem}[section]
\newtheorem{thm}{Theorem}[section]
\newtheorem{prop}[theorem]{Proposition}

\newtheorem{lemma}[theorem]{Lemma}
\newtheorem{cor}[theorem]{Corollary}

\theoremstyle{definition}

\newtheorem{exa}[theorem]{Example}

\newtheorem{rem}[theorem]{Remark}

\numberwithin{equation}{section}

\begin{document}

\title{On Jacobians of curves with superelliptic components}

\author{L. Beshaj}
\address{Department of Mathematics, Oakland  University, Rochester, MI, 48309}

\email{beshaj@oakland.edu}
\thanks{}

\author{T. Shaska}
\address{Department of Mathematics, Oakland  University, Rochester, MI, 48309}
 
\email{shaska@oakland.edu}
\thanks{}

\author{C. Shor}
\address{Department of Mathematics, Western New England University, Springfield, MA 01119}
 
\email{cshor@wne.edu}
\thanks{}



\maketitle

\begin{abstract}
We construct a family of non-hyperelliptic curves whose Jacobians decomposes into a product of superelliptic Jacobians. This generalizes a construction in \cite{Ya}. 
Moreover,  we investigate the decomposition of Jacobians of superelliptic curves based on their automorphisms. For a curve given by the equation $y^n=f(x^m)$, we provide a necessary and sufficient condition in terms of $m$ and $n$ for the  the Jacobian of the curve to decompose.

\end{abstract}

\section{Introduction}
Let $\X$ be a genus $g\geq 2$ smooth, irreducible projective curve, defined over an algebraically closed  field $k$ and $\pi_i : \X \to \X_i$, $1 \leq i \leq s$ be coverings to genus $g_i$ curves $\X_i$, respectively.  Denote by $\Jac  (\X)$ and $\Jac  (\X_i)$ the Jacobians of these curves respectively.  A classical question is to determine when the Jacobian $\Jac  (\X)$ is isogenous to a product of $\Jac  (\X_i)$. 
The problem has been studied in the XIX century by Legendre, Jacobi, Hermite, Klein, Kovalevskaya, Hecke, et al.  In this context, these curves were of interest due to the fact that split Jacobians give relations between abelian and elliptic integrals. Most of these efforts were focused on the case when the coverings $\pi_i: \X \to \X_i$ are Galois coverings. 

The case when genus $g(\X_i)=1$ for some $i$ and $\pi_i : \X \to \X_i$ not necessarily Galois  is especially interesting due to its implications in number theory and relations to  Lang's conjecture; see \cite{faltings-1, faltings-2} and \cite{martens}. When $g=2$ such coverings  have different ramification structures when $\deg \pi_i$ is odd or even; see \cite{b-sh} for a summary of results in genus two.   In both cases the moduli space of such covers with  fixed ramification structure can be embedded as a subvariety of $\M_2$.   

The Jacobian $\Jac  (\X)$  is said to (completely) split if it is isogenous to the product of elliptic curves.  Ekedahl and Serre  \cite{Se} posed the following questions: 
i) Is it true that for every integer $g > 1$ there exists a curve of genus $g$ with split Jacobian? 
ii) Is there a bound on the genus of a curve with split Jacobian? 
Towards answering these questions they used modular curves, and coverings of curves of genus 2 and 3.  The maximal genus of their examples was 1297. Many other authors have considered these questions as well, and in recent work, Shaska \cite{arxiv-1, arxiv-2} and  Yamauchi \cite{Ya}, have produced somewhat unusual families of curves with decomposable Jacobians.

In this paper we will focus on the cases when the coverings $\pi_i : \X \to \X_i$ are cyclic Galois coverings. Hence, these coverings induce automorphisms on $\X$.  By focusing on such curves  we have a more organized way to study the decompositions of such Jacobians starting from the automorphism group $\Aut (\X)$ of $\X$.  While we can determine the full list of groups that occur as full automorphism groups of genus $g \geq 2$, the corresponding equations of curves for each group are not known; see \cite{dec_jac}.   There is, however, a nice family of curves when such decomposition can be fully explored and the factors of such Jacobians can be fully determined. This is the class of superelliptic curves, namely the curves which can be written as affine varieties with equation $y^n = f(x)$ for some $n\geq 2$ and discriminant $\Delta (f, x) \neq 0$.  For such curves we know precisely the structure of the automorphism group $G$, the signature, the equations of the curves, and their invariants; see \cite{beshaj-2}.    
Hence, it is possible  to determine explicitly the components of such Jacobians and in some cases the moduli space of such curves in the moduli space $\M_g$. 
A complete description of curves for which such  decompositions is based on their automorphisms is intended in \cite{dec_jac}.

The second goal of the paper is to extend the family of curves introduced by Yamauchi in \cite{Ya}.  The family of curves introduced in \cite{Ya} are non-hyperelliptic curves with arbitrary large genus such that the Jacobian has hyperelliptic components.   We attempt to extend this family to curves whose Jacobians have superelliptic components. Our proof   is based on the automorphisms of curves.  We show that such decomposition based on the automorphisms of the curve induces some arithmetic condition among the orders of such automorphisms.  This condition implies that the family $F_{m, n}$ constructed in \cite{Ya} agrees with our family of curves only for the values $F_{m, 1}$ and $F_{m, 2}$.  

This paper is organized as follows. In section 2 we briefly define the Jacobian of a curve and give two classical results of Accola; see \cite{Ac1, Ac2}.  In these results it is described how one can start from a partition of a group of automorphisms of the curve $\X$  and get a decomposition of the  $\Jac  (\X)$.  Further in this section we describe results of Kani and Rosen were the previous results were generalized; see \cite{KR}.  We use such results   in our decomposition of Jacobians of the superelliptic curves  for the family of curves described in section 4. 

In section 3 we study the decomposition of Jacobians of superelliptic curves.  A superelliptic curve $\X$ is a curve with equation $y^n = f(x)$ defined over a field $k$, where $\left( \ch k, n \right)=1$ and $k$ is algebraically closed.   Such curves have the \emph{superelliptic} automorphism $\tau : (x, y) \to (x, \e_n y)$, where $\e_n$ is an $n$-th primitive root of unity and $\tau$ is central in $G:=\Aut (\X)$. 
Let $\s$ be another automorphism of $\X$ such that its projection $\bs \in \bG:=G/\< \t \>$ has order $m$.   Then $\s$ and $\s\t$ fix two subfields of the function field $k (\X)$ and therefore there are two quotient curves $\X_1:= \X /\< \s\>$ and $\X_2:= \X/ \< \s\t\>$.  We determine equations of $\X_1$ and $\X_2$ and a necessary and sufficient condition in terms of $m$ and $n$ such that $\Jac  (\X)$ is isogenous to $\Jac  (\X_1) \times \Jac  (\X_2)$. 

In section 4 we generalize a construction of Yamauchi of a family $F_{m, n}$ of non-hyperelliptic curves with decomposable Jacobians, where all components are hyperelliptic Jacobians.  Instead we construct a family $\X_{r, s}$ of curves whose Jacobians are superelliptic Jacobians.  We prove that the automorphism group of these \emph{component curves} are  cyclic or dihedral groups. Moreover, we find a necessary and sufficient condition in terms of $r$ and $s$ for this decomposition to occur. Our decomposition is based solely on the decomposition induced by the automorphisms of the curves. 

\medskip

\noindent \textbf{Notation:} Throughout this paper by $g$ we denote an integer $\geq 2$ and  $k$ denotes an algebraically closed field.  Unless otherwise noted,  by a \textit{curve} we always mean the isomorphism class of an algebraic curve defined over $k$.   The automorphism group of a curve always means the full automorphism group of the curve over $k$.  We denote the cyclic group of order $n$ by $C_n$ and the dihedral group of order $2n$ by $D_{2n}$.  $V_4$ denotes the Klein 4-group and $\e_n$ denotes an $n$-th primitive root of unity.

\section{Preliminaries}

Let $\X$ be a genus $g \geq 2$ algebraic curve defined over $\C$. We choose a symplectic homology basis for $\X$, say $ \{ A_1, \dots, A_g, B_1, \dots , B_g\},$ such that the intersection products $A_i \cdot A_j = B_i \cdot B_j =0$ and
$A_i \cdot B_j= \d_{i j},$ where $\d_{i j}$ is the Kronecker delta. We choose a basis $\{ w_i\}$ for the space of holomorphic 1-forms such that $\int_{A_i} w_j = \d_{i j}$. The matrix $\O= \left[ \int_{B_i} w_j
\right] $ is  the \emph{Riemann matrix} of $\X$ and the matrix $\left[ I \ | \O \right]$ is called the \textit{period matrix}.  The columns of the matrix $\left[ I \ | \O \right]$ form a lattice $L$ in  $\C^g$.  The complex torus  $\C^g/ L$ is called the Jacobian  of $\X$ is denoted by $\Jac  (\X)$. 

Let $\H_g$ be the \emph{Siegel upper-half space}. Then $\O \in \H_g$ and there is an injection
\[ \M_g \embd \H_g/ Sp_{2g}(\Z) =: \A_g \] where $Sp_{2g}(\Z)$ is the \emph{symplectic group}.   
A non-constant morphism  $f : A \to B$ between two Abelian varieties which is surjective and of finite kernel is called an \textbf{isogeny}. An Abelian variety is called \textit{decomposable} if it is isogenous to a product of Abelian varieties, it is \textit{simple} if it has no non-trivial Abelian subvarieties. 
An Abelian variety is called \textbf{completely decomposable} or \textbf{completely split}   if it is isogenous to a product of elliptic curves. 

A map of algebraic curves $f: \X \to \Y$ is called a \textbf{maximal covering} if it does not factor over a nontrivial isogeny.  A map $f: \X \to \Y$ induces maps between their Jacobians $f^*: \Jac  (\Y) \to \Jac  (\X)$ and $f_*: \Jac  (\X) \to \Jac  (\Y)$. When $f$ is maximal then $f^*$ is injective and $\ker (f_*)$ is connected, see \cite[p. 158]{Se-book} for details.  Hence, $\Jac  (\X) \iso \Jac  (\Y) \times A$, where $A$ is some Abelian variety.  Thus, coverings $f: \X \to \Y$ give factors of the Jacobian $\Jac  (\X)$.   Such methods have been explored for genus 2 curves by Shaska et al. in \cite{sh_01, deg3, deg5}.  They are the only examples that we know when explicit computations have been performed and the corresponding locus has been computed for non-Galois coverings. 

If the covering $f: \X \to \Y$ is a Galois covering then its monodromy group is isomorphic to   a subgroup $H$ of the automorphism group $G=\Aut (\X)$.  Hence a common procedure to produce decompositions of Jacobians is to explore the automorphism group of the curve.  

\def\bC{\textbf{C}}

Fix an integer $g\ge2$ and a finite group $G$. Let $C_1, \dots ,C_r$ be conjugacy classes $\ne\{1\}$ of $G$. Let $\textbf{C}=(C_1, \dots ,C_r)$, be an unordered tuple, repetitions are allowed. We allow $r$ to be zero, in which case $\bC$ is empty.

Consider pairs $(\X,\mu)$, where $\X$ is a curve and $\mu: G \to \Aut(\X)$ is an injective homomorphism. Mostly we will suppress $\mu$ and just say $\X$ is
a curve with $G$-action, or a $G$-curve, for short. Two $G$-curves $\X$ and $\X'$ are called equivalent if
there is a $G$-equivariant isomorphism $\X\to \X'$.

We say a $G$-curve $\X$ is \textbf{of ramification type} $(g, G,\bC)$ if  $g$ is the
genus of $\X$ and  the points of the quotient $\X/G$ that are ramified in the cover $\X\to \X/G$ can be
labelled as $p_1, \dots , p_r$ such that $C_i$ is the conjugacy class in $G$ of distinguished inertia group
generators over $p_i$ (for $i=1, \dots ,r$).  

Let $\X$ be a $G$-curve and $H < G$. Then there is a covering $\X \to \X/H$. Let the genus of $\X/H$ be denoted by $g_H$.  How is $g_H$ determined in terms of $g$ and $G$?  
Consider the following problem: let $H_1, \dots , H_r$ be subgroups of $G$ and $g_1, \dots , g_r$ the genera of the $\X/H_1 , \dots , \X/H_r$ respectively.  Is there any arithmetic relation between $g, g_1, \dots , g_r$?

Accola proved the following results which provide a method for decomposing Jacobians, see \cite{Ac1}.
\begin{thm}[\cite{Ac1}]
Let $\X_g$ be a $G$-curve, $H_i < G$, $g_i$ the genus of $\X_g/H_i$ for $i=1, \dots s$, and $g_0 = g(\X_g/G)$. Assume that $H_i \cap H_j = \{ e \}$ for all $i\neq j$.  Then,
\[  g_0 \, |G| = g - sg + \sum_{i=1}^s |H_i | \, g_i.   \]
\end{thm}
\noindent Moreover, we have the following relation among the genera; see \cite{Ac2}.
\begin{thm}[\cite{Ac2}]
Let $\X_g$ be a $G$-curve, and $H_i$, for $i=1, \dots s$, subgroups of  $G$ such that $G= \cup_i^s H_i$.  Denote by  $g_i$ the genus of $\X_g/H_i$ and by $H_{ij...k}=H_i \cap H_j \cap \dots \cap H_k$.  

 Then,
\[ g_0 \, |G| = \sum_{i=1}^s |H_i | \cdot g_i - \sum \left|H_{ij} \right| \cdot g_{ij} + \sum \left|H_{ijk}\right| \cdot g_{ijk} - \cdots - (-1)^s \sum \left|H_{12...s} \right| \cdot  g_{12..s}. \]

\end{thm}

See \cite{Ac2} for details.

\subsection{Decomposing the Jacobian by group partitions}

Let $\X$ be a genus $g$  algebraic curve with automorphism group $G:=\Aut (\X)$. Let $H \leq G$ such that $H = H_1 \cup \dots \cup H_t$ where the subgroups $H_i \leq H$ satisfy $H_i \cap H_j = \{ 1\}$ for all $i\neq
j$.  Then,
\[ \Jac^{t-1} (\X ) \times \Jac^{|H|} (\X / H)\,  \iso \, \Jac^{| H_1 |} (\X / H_1) \times \cdots \Jac^{| H_t | } (\X / H_t).\]
The group $H$ satisfying these conditions is called a group with partition. Elementary Abelian $p$-groups, the projective linear groups $PSL_2 (q)$, Frobenius groups, dihedral groups are all groups with partition.

Let $H_1, \dots , H_t \leq G$ be subgroups with $H_i \cdot H_j = H_j \cdot H_i$ for all $i, j \leq t$, and let $g_{ij}$ denote the genus of the quotient curve $\X/(H_i\cdot H_j)$. Then, for $n_1, \dots , n_t \in \Z$
the conditions    $ \sum n_i n_j g_{ij} =0$, $\sum_{j=1}^{t} n_j g_{ij}=0$,  imply the isogeny relation
\begin{equation} \label{dec_1}
\prod_{n_i > 0} \Jac^{n_i} (\X / H_i) \iso \prod_{n_j < 0} \Jac^{n_j} (\X / H_j).
\end{equation}
In particular, if $g_{ij}=0$ for $2\leq i < j \leq t$ and if   $ g = g_{\X/ H_2} + \dots + g_{\X / H_t}$,  then
\begin{equation} \label{dec_2}
\Jac  (\X) \iso \Jac  (\X /H_2) \times \cdots \times \Jac  (\X / H_t).
\end{equation}
The proof of the above statements can be found in \cite{KR}.   

\section{Jacobians of superelliptic curves} 

A curve $\X$ is called superelliptic if there exist an element $\tau \in G:=\Aut( \X)$ such that $\tau$ is central and $g \left(\X / \< \tau \> \right) =0$. Let $\X_g$ have affine equation  given by some polynomial in terms of $x$ and $y$ and denote by $K=k(x, y)$ the function field of $\X_g$.  Let $H=\< \tau \>$ be a cyclic subgroup of $G$ such that $| H | = n$ and $H \normal G$, where $n \geq 2$. Moreover, we assume that the quotient curve $\X_g / H$ has genus zero.   The \textbf{reduced automorphism group of $\X_g$ with respect to $H$} is called the group  $\G \, := \, G/H = \bAut (\X)$.  

Assume  $k(x)$ is the  genus zero subfield of $K$ fixed by $H$.   Hence, $[ K : k(x)]=n$. Then, the group  $\G$ is a subgroup of the group of automorphisms of a genus zero field.   Hence, $\G <  PGL_2(k)$ and $\G$ is finite. It is a classical result that every finite subgroup of $PGL_2 (k)$  is  isomorphic to one of the following: $C_m $, $ D_{2m}$, $A_4$, $S_4$, $A_5$.

The group $\G$ acts on $k(x)$ via the natural way. The fixed field of this action is a genus 0 field, say $k(z)$. Thus, $z$
is a degree $|\G|$ rational function in $x$, say $z=\phi(x)$.  

It is obvious that $G$ is a degree $n$ central extension of $\G$ and $\G$ is a finite subgroup of $PGL_2(k)$.  Hence, if we know all the possible groups that occur as $\G$ then we can determine  $G$ and the equation for $K$; see \cite{Sa2}. 
In the next Lemma we establish some basic properties of superelliptic curves, the proof of the part i)  can be found in \cite{tw1} and part ii)  in \cite{beshaj-2}.

\begin{lemma} \label{lem_1}
Let $\X_g$ be a superelliptic curve  with affine equation $y^n= f(x)$ where $\Delta (f,x) \neq 0$ and $\deg f =d >n$.  Then the following hold:

i) $\X_g$ has genus \[ g = 1 + \frac 1 2 \, \left( nd -n -d - \gcd (d, n) \right). \]
If $d$ and $n$ are relatively prime then $g = \frac {(n-1)(d-1)} 2$. 

ii) Let  $\s \in \Aut (\X_g)$ such that its projection $\bs \in \bAut (\X_g)$ has order  $m\geq 2$.  Then its equation is given as $y^n= g(x^m)$ or $y^n=x g(x^m)$ for some $g \in k[x]$. 
\end{lemma}

Let $\X_g$ be a superelliptic curve and $\s \in \Aut (\X_g)$ such that its projection $\bs \in \bAut (\X_g)$ has order  $m\geq 2$.  We can choose a coordinate in $\P^1$ such that $\bs (X) = X^m$. Since $\s$ permutes the Weierstrass points of $\X$ and it has two fixed points then the equation of the curve will be $y^n= f(x^m)$ or $y^n=x f(x^m)$  as claimed in ii) above.  

Assume that $\X$ has equation 
\begin{equation}\label{first_eq}
 Y^n = f(X^m) := X^{\d m } + a_1 X^{(\d-1) m}  \dots + a_{\d-1} X^m + 1 .
\end{equation} 
We assume that $\bs$ lifts to $G$ to an element of order $m$.   Then, $\s( X, Y) \to (\e_m X, Y)$.  Denote by $\t: (X, Y) \to (X, \e_nY)$ its superelliptic automorphism. Since $\t$ is central in $G$ then $\t\s=\s\t$. 
We will denote by $\X_1$ and $\X_2$ the quotient curves $\X/\<\s\>$ and $\X/\< \t \s\>$ respectively.  The next theorem determines the equations of $\X_1$ and $\X_2$.  We denote by $K$ the function field of $\X_g$ and by  $F$ and $L$ the function fields of $\X_1$ and $\X_2$ respectively. 

\begin{thm}\label{thm_1}
Let $K$ be a genus $g\geq 2$  level $n$ superelliptic field and $F$ a  degree $m$ subfield fixed by $\s: (X, Y) \to (\e_m X, Y)$. 

i) Then,  $K=k(X, Y)$ such that 
\begin{equation}\label{first_eq}
 Y^n = f(X^m) := X^{\d m } + a_1 X^{(\d-1) m}  \dots + a_{\d-1} X^m + 1 .
\end{equation} 
for    $\D (f, x)  \neq 0$.

ii) $F = k(U, V) $ where $U=X^m$,    and $V=Y$  and 
\begin{equation}\label{ell_curve}
 V^n= f(U).
 \end{equation}

iii) There is another subfield $L= k(u, v) $ where  \[ u=X^m,  \quad v=X^i Y, \] and 
\begin{equation}\label{gen_2}
 v^n = u \cdot f(u),
 \end{equation}
for $m=\l n$ and $i=\l (n-1)$. 
\end{thm}

\proof
 The proof of i) follows from the above remarks.  To show that the   subfield  $F$ is generated by $X^m$ and $ Y$ it is enough to show that it is  fixed by $\sigma$.  This is obvious. 
 
In  iii) we have to show that $u=X^m$ and $ v=X^i Y$ are fixed by $\s\t : (X, Y) \to (\e_m X, \e_n Y)$.  We have that 
$ \s \t (v) =  \s \t (X^i Y) = \e_m^{\l (n-1)} \e_n \cdot X^i Y = X^i Y$. 
It is easily checked that $\e_m^{\l (n-1)} \e_n=1$.
 \qed

\begin{center}
\begin{figure}[htpb]
\[
\xymatrix{                           
                     & K=k(X, Y) \ar@{->}[d]^{\, \, \,n}  \ar@{->}[ld]^{\, \, \,m}  \ar@{->}[rd]^{\, \, \, \, \lcm(m, n)}           &                  \\
F=k(X^m, Y)   \ar@{->}[d]^{\, \, \,n}   & k(X)  \ar@{->}[ld]^{\, \, \,m}  \ar@{->}[rd]^{\, \, \, \lcm(m, n)}  & L=k(X^m, X^iY)   \ar@{->}[d]^{\, \, \,n}   \\
k(X^m, Y^n)       &     & k( X^m, (X^i Y)^n)       \\
}
\]
\caption{The lattice of a genus $g$ superelliptic field with an extra automorphism}
\end{figure}
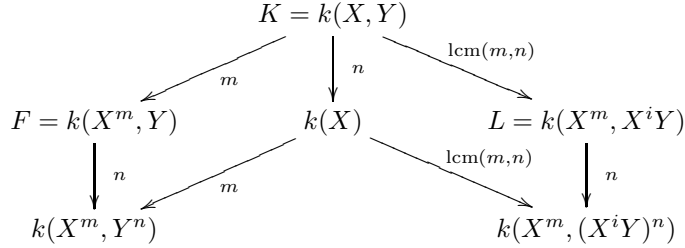
\end{center}

In the rest of this section we want to find necessary and sufficient conditions on $n$ and $m$ such that the Jacobian $\Jac  (\X_g)$ is isogenous to the product $\Jac  (\X_1) \times \Jac  (\X_2)$. First we focus on  hyperelliptic curves.

\subsection{Hyperelliptic curves}

\begin{thm} \label{lem_2}
Let $\X_g$ be a hyperelliptic curve and its reduced automorphism group $ \bAut (\X_g) \iso  C_m = \<\sigma \>$. Then $\X_g$ is isomorphic to a curve with equation
\[\X_g: Y^2=x^{\d  m} + a_1 x^{(\d-1)  m} + \dots + a_{\d-1} x^m + 1.\]
There exists subcovers $\pi_i : \X_g \to \X_i$, for $i=1, 2$ such that 
 \[ 
 \begin{split}
&  \X_1 : \quad Y^2= X^{\d} + a_1 X^{\d-1} + \dots + a_{\d-1} X + 1, \\
& \X_2 : \quad Y^2=X(  X^{\d} + a_1 X^{\d-1} + \dots + a_{\d-1} X + 1  ). \\
\end{split}
\] 
The  Jacobian of $\X$ is isogenous to the product
\[ \Jac  (\X) \iso  \Jac  (\X_1)   \times   \Jac  (\X_2) \]
%
if and only if the full automorphism group $\Aut (\X)$ is isomorphic to the Klein 4-group $V_4$. 
\end{thm}

\proof Assume $\X_g$ is a hyperelliptic curve and $\sigma \in \bar G \iso C_m$.  The equation of the hyperelliptic curve $\X_g$ is
\[\X_g: Y^2=x^{\d  m} + a_1 x^{(\d-1)  m} + \dots + a_{\d-1} x^m + 1.\]
Denote with  $\tau$ the hyperelliptic involution. We have an extra involution $\sigma \in \bar G \iso C_m$ and since extra involutions come in pairs we have $\sigma \tau \in \bar G$. The two fixed spaces of $\sigma$ and $\sigma \tau$ have equations respectively
 \[ \X_1 : \quad Y^2= X^{\d} + a_1 X^{\d-1} + \dots + a_{\d-1} X + 1\]
and
\[ \X_2 : \quad Y^2=X(  X^{\d} + a_1 X^{\d-1} + \dots + a_{\d-1} X + 1  ),\] 
where $X=x^m$.  The genera $g_1, g_2$ are  respectively $[ \frac{\d-1}2]$ and $[ \frac{\d}2].$ Assume $\d$ is even, say $\d=2k$, then we have $g= g_1+g_2=  \d-1$. When $\d$ is odd, say $\d=2k+1$ we have  $g= g_1+g_2 =\d-1$ and from \cite[Table 3]{beshaj-2}, $\d$ is as follows
$ \frac{2(g+1)}{m}$,  $\frac{2g+1}{m}$, $ \frac{2g}{m}-1$ 
and $m$ respectively   $\frac{2(g+1)}{\d}$, $ \frac{2g+1}{\d+1}$, $ \frac{2g}{\d+1}$.

Considering each case we get $m=2$ in the first case and $m$ not an integer in the other two.  Therefore, $G \iso V_4$.

Conversely, assume $G \iso V_4$. The equation of the hyperelliptic curve is 
\[\X_g: Y^2=x^{2g+2} + a_1 x^{2g} + \dots + a_g x^2 + 1. \]  
There are two extra involutions $\sigma$ and $\sigma \tau$ in $G$ such that $|\sigma|= |\sigma \tau| = 2$. They fix the following curves:
 \[ \X_1 : \quad Y^2= X^{g+1} + a_1 X^{g} + \dots + a_{g} X + 1\]
and \[ \X_2 : \quad Y^2=X(  X^{g+1} + a_1 X^{g} + \dots + a_{g} X + 1  ),\] 
where $X=x^2$. If we evaluate the genera we have respectively    $g_1=\left[\frac{(g+1)-1}2 \right] = \left[\frac{g}2 \right]$ and  $g_2=\left[\frac{(g+2)-1}2 \right] = \left[\frac{g+1}2 \right]$.  Thus $g=g_1 +g_2$.   Therefore,   Jacobian is isogenous to 
\[ \Jac  (\X) \iso  \Jac  (\X_1)   \times   \Jac  (\X_2). \]
This completes the proof. 

\qed

\subsection{Non-hyperelliptic curves}

\noindent Next, we generalize the previous theorem.

\begin{thm}\label{thm:superelliptic}
 Let $\X_g$ be a level $n$ superelliptic curve and $C_m = \< \bar \s \> \embd \bAut (\X_g)$, where  $m \geq 2$ and  the equation of $\X_g$ is  $ y^n = f(x^m)$,
with $\deg (f) = d = \d\, m$,  $d > n$.  Then there exist degree $m$ coverings $\pi : \X_g \to \X_i$, $i=1, 2$  where 
\[  \X_1 : \quad Y^n=f(X) \quad and  \quad \X_2 : \quad  Y^n= X f(X).  \]
Then,  \[ \Jac  (\X) \iso \Jac  (\X_1 )  \times \Jac  ( \X_2 ) \] 
if and only if 
\begin{equation}\label{eq_m}
  \d (n-1) (m-2) = 1 -   \left(\gcd(\d+1, n) + \gcd(\d, n) - \gcd(\d m, n)    \right).
\end{equation} 
\end{thm}

\proof
 Let $\X_g$ be a superelliptic curve with and extra automorphism of order $m \geq 2$ and equation $y^n = f(x^m)$.  There is the superelliptic automorphism 
\[ \tau :   \, \,   (x, y) \to (x, \e_n y), \quad and \quad \bar \sigma :   \, \,  (x, y) \to (\e_m x, y) .  \]
   We denote by $\sigma$ the lifting of $\bar \sigma$ in $\Aut (\X)$.   Then, $\sigma\tau=\tau\sigma$.   

Let  $H_1 : =\<  \sigma \> $ and  $H_2 : = \< \s \tau \>$ be subgroups in $G$.  Then, $| H_1 | = n$ and $ | H_2 | = \lcm (n, m)$.    Thus, we have $H:=H_1 \times H_2 \embd G$. It is easy to check that $g \left(\X_g / (H_1 H_2) \right)=0$.

Moreover,    $\sigma $ and $\sigma \tau$ fix the curves 
 \[ \X_1 : \quad Y^n= X^{\d} + a_1 X^{\d-1} + \dots + a_{\d-1} X + 1\]
and \[ \X_2 : \quad Y^n=X(  X^{\d} + a_1 X^{\d-1} + \dots + a_{\d-1} X + 1  ).\] 
Let $g_1$ and $g_2$ denoted their genera respectively.  From Lemma 3.1 we have  that
\[ g_1 = 1 + \frac 1 2 \left( n \d - n - \d - \gcd(\d, n)   \right)  \]
and 
\[ g_2 = 1 + \frac 1 2 \left( n (\d+1) - n - (\d + 1)  - \gcd(\d + 1, n)   \right).  \]
Then we have 
\[ g_1 + g_2 =  \frac 3 2 + n\d - \frac n 2 - \d - \frac 1 2 \left( \gcd( \d, n) + \gcd(\d+1, n) \right). \]
The genus of $\X$ is 
 \[ g = 1 + \frac 1 2 \left(  n \d m - n - \d m - \gcd(m \d, n)  \right). \]
Then, $g=g_1+g_2$ implies that 
\[  \d  (n-1) (m-2) = 1 -   \left(\gcd(\d+1, n) + \gcd(\d, n) - \gcd(\d m, n)    \right). \]
Thus,  from Eq.~\eqref{dec_2} we have that 
\[ \Jac  (X_g) \iso \Jac  (\X / H_1) \times \Jac  (\X / H_2) \]
which completes the proof. 

\endproof

\begin{cor}
 Let $\X_g$ be a level $n$ superelliptic curve as in Theorem~\ref{thm:superelliptic}.  Furthermore, suppose $n$ is prime.  Then $m=2$ or $m=3$.  

In particular, one of the following situations is true:
\begin{itemize}
\item $n$ is any prime, $m=2$, and $\delta \equiv 0 \text{ (mod $n$)}$;
\item $n=2$, $m=2$, and $\delta$ is odd;
\item $n=3$, $m=3$, $\delta=1$;
\item $n$ is any odd prime, $m=2$,  $\delta \not\equiv 0,-1 \text{ (mod $n$)}$.
\end{itemize}
\end{cor}

\begin{proof}
If $n$ is prime, then the gcd's on the right hand side of the Eq.~\eqref{eq_m} are each either $1$ or $n$.  We consider cases below.  Note that at least one of $\gcd(n,\delta)$ and $\gcd(n,\delta+1)$ equals 1 because $n$ is prime.

If $\gcd(n,\delta)=n$, then $\gcd(n,m\delta)=n$, so \[1 + \gcd(n,m\delta) - \gcd(n,\delta) - \gcd(n,\delta+1) = 1+n-n-1=0.\]  Thus, $\delta(m-2)=0$.  But $\delta\neq0$, so $m=2.$  And since $\gcd(n,\delta)=n$, then $\delta\equiv0\text{ (mod $n$)}$.

If $\gcd(n,\delta+1)=n$, then $\gcd(n,\delta)=1$, so $\gcd(n,m\delta)=\gcd(n,m)$.  Thus, \[1 + \gcd(n,m\delta) - \gcd(n,\delta) - \gcd(n,\delta+1) = \gcd(n,m)-n.\]  Since we have $m\geq2$, we know $\delta(n-1)(m-2)\geq0$, so $\gcd(n,m)-n\geq0$, so $\gcd(n,m)\geq n$, so $\gcd(n,m)=n$.  Thus, \[1 + \gcd(n,m\delta) - \gcd(n,\delta) - \gcd(n,\delta+1) = 0,\] so $\delta(n-1)(m-2)=0$, so we again conclude $m=2$.  Since $\gcd(n,m)=n$ with $n$ prime and $m=2$, we get $n=2$.  And since $\gcd(n,\delta+1)=n$, $\delta$ is odd.

If $\gcd(n,\delta)=\gcd(n,\delta+1)=1$, then we consider $\gcd(n,m)$. If $\gcd(n,m)=1$, then again \[1 + \gcd(n,m\delta) - \gcd(n,\delta) - \gcd(n,\delta+1) = 0,\] so $m=2$ and $\delta\not\equiv 0,-1\text{ (mod $n$)}$.
If $\gcd(n,m)=n$, then \[1 + \gcd(n,m\delta) - \gcd(n,\delta) - \gcd(n,\delta+1) = n-1,\] so $\delta(m-2)=1$, so $\delta=1$ and $m=3$.  Since $\delta=1$ and $\gcd(n,\delta+1)=1\neq2$, this implies that $2\nmid n$, so $n$ is an odd prime.
\end{proof}

\begin{exa}
For triagonal curves, $n=3$.  Then either i) $m=3$ and $\delta=1$; or ii) $m=2$ and $\delta\equiv0,1\text{ (mod 3)}$.
\end{exa}

\section{Jacobians with superelliptic components}
In this section we study a family of non-hyperelliptic curves introduced in \cite{Ya}, whose    Jacobians decompose into factors which are  hyperelliptic Jacobians.  We will extend this family of curves and investigate if we can obtain in this way curves of arbitrary large genus having decomposable Jacobians. 

In \cite{Ya}  were introduced a family of curves  in $\P^{s+2}$ given by the equations 
\begin{equation}\label{curve}
\left\{
\begin{split}
zw & = c_0x^2+c_1xw+c_2w^2 \\
y_1^r & = h_1 (z, w) := z^r+ c_{1, 1} z^{r-1} w + \cdots + c_{r-1, 1} zw^{r-1} + w^r,\\
 & \dots  \\
y_s^r & = h_s (z, w) :=  z^r+ c_{1, s} z^{r-1} w + \cdots + c_{r-1, s} zw^{r-1} + w^r,  \\
\end{split}
\right.
\end{equation}
where  $c_i \in k$, $i=0, 1, 2,$ and   $c_{i,j} \in k$ for $i=1, \cdots, r$, $ j=1, \cdots , s$. 
The variety $\X_{r, s}$ is an algebraic curve since the function field of $\X_{r, s}$ is a finite extension of $k(z)$. $\X_{r, s}$ is a complete intersection. 
For a proof of the following facts see  \cite{Ya}.

\begin{rem}Let $\X_{r, s}$ be as above. Assume that $\X_{r, s}$ is smooth and $c_0\neq 0$. Then  the following hold: 

i) The genus of  $\X_{r, s}$ is  
\[ g(\X_{r, s}) = (r-1)(rs\cdot2^{s-1}-2^s+1).\] 

ii) If $r\geq3$ and $s\geq1$, then $\X_{r, s}$ is non-hyperelliptic.
\end{rem} 

Fix $r\geq 2$.  Let $\l$ be an integer such that $1  \leq \l\leq s$.    Define the superelliptic curve $C_{r, \l, m}$ as follows
\[ C_{r, \l, m}:  \qquad Y^r=\prod_{i=1}^{\l} h_{i}(X^m, 1),\]
for some $m\geq 2$.  The right side of the above equation has degree $d = rm \l$. Using  Lemma~\ref{lem_1} we have that 
\[ g( C_{r, \l, m} ) = 1 + \frac 1 2 \left(  r^2 m\l - r- m\l r - \gcd(\l r m, r)\right). \]
Hence,
\begin{equation} \label{eq_g}
g( C_{r, \l, m} ) = 1 + \frac r 2 \left((r-1)\l m -2  \right).
\end{equation}

\begin{thm}
Let $C_{r, \l, m}$ be a generic curve as above. Then the following hold

i) $\bAut \left( C_{r, \l, m} \right) \iso C_m$.

ii) $\bAut \left( C_{2, \l, m} \right) \iso D_{2m}$.

\end{thm}

\proof  i) The equation of $C_{r, \l, m}$ is $y^r= g(x^m)$. Then the curve has the following two automorphisms 
\[ \tau (X, Y) \to (X, \e_r Y), \quad \sigma (X, Y) = (\e_m X, Y).\]
Since $\tau$ commutes with all automorphisms then $\< \tau \> \normal \Aut (C_{r, \l, m} )$ and $C_{r, \l, m} / \< \tau \>$ is a genus 0 curve.  Then $\< \sigma \>  \embd \bAut (C_{r, \l, m}) $. If $C_{r, \l, m}$ is a generic curve then 
$\< \sigma \>  \iso \bAut (C_{r, \l, m}) $.

ii)  When $r=2$ then the equation of $C_{2, \l, m}$ is given by 
\[ y^r = \prod_{i=1}^\l (X^{2m} + a_i X^m + 1). \]
From \cite{beshaj-2} such curves have reduced automorphism group isomorphic to $D_{2m}$. This completes the proof. 
\qed

The following theorem determines the full automorphism group of such curves.

\begin{thm}
Let $C_{r, \l, m}$ be a generic curve as above and G:=$\Aut \left( C_{r, \l, m} \right)$. Then the following hold \\

\noindent a) If    $\bAut \left( C_{r, \l, m} \right) \iso C_m $, then $G \cong C_{mn}$ or   $G$ is isomorphic to 
\begin{center}
$\left\langle \r, \s \right|\r^n=1,\s^m=1,\s\r\s^{-1}=\r^l \rangle$
\end{center}
where (l,n)=1 and $l^m\equiv 1$ (mod n). But if $(m,n)=1$, then $l=n-1$. \\

\noindent  b) If    $\bAut \left( C_{r, \l, m} \right) \iso D_{2m} $, then  \\

(1) If n is odd then $G \cong D_{2m} \times C_n$.

(2) If n is even and m is odd then $G \cong D_{2m} \times C_n$ or $G$ is isomorphic to the group with presentation 
\[ \left\langle \r, \s, \t \right|\r^n=1,\s^2=\r,\t^2=\r^{n-1},(\s\t)^m=\r^{\frac{n}{2}},\s\r\s^{-1}=\r,\t\r\t^{-1}=\r \rangle. \]

(3) If n is even and m is even then $G $ is isomorphic to one of the following groups $D_{2m} \times C_n$, $D_{2mn}$, or one of the following 
\begin{align*}
\begin{split}
G_1=& \left\langle \r, \s, \t \right|\r^n=1,\s^2=\r,\t^2=1,(\s\t)^m=1,\s\r\s^{-1}=\r,\t\r\t^{-1}=\r^{n-1} \rangle,\\
G_2=& \left\langle \r, \s, \t \right|\r^n=1,\s^2=\r,\t^2=\r^{n-1},(\s\t)^m=1,\s\r\s^{-1}=\r,\t\r\t^{-1}=\r \rangle,\\
G_3=& \left\langle \r, \s, \t \right|\r^n=1,\s^2=\r,\t^2=1,(\s\t)^m=\r^{\frac{n}{2}},\s\r\s^{-1}=\r,\t\r\t^{-1}=\r^{n-1} \rangle, \\
G_4=& \left\langle \r, \s, \t \right|\r^n=1,\s^2=\r,\t^2=\r^{n-1},(\s\t)^m=\r^{\frac{n}{2}},\s\r\s^{-1}=\r,\t\r\t^{-1}=\r \rangle.
\end{split}
\end{align*}
\end{thm}

\proof The full automorphism group $G$ of such curves is a degree $n$ central extension of $C_m$ or $D_{2m}$.  Such extensions were determined 
in \cite[Thm.~3.2]{Sa1} and \cite[Thm.~3.3]{Sa1}.

\qed

\subsection{Decomposition of Jacobians}

In \cite{Ya} it is proved that the Jacobians of $C_{2, \l, 2}$, defined over any algebraic number field $k$,  curves are isogenous to a product of superelliptic Jacobians. A similar theorem was suggested for curves $C_{r, \l, 2}$ and was remarked that the proof was similar to the case $C_{2, \l, 2}$ curves. The proof is of arithmetic in nature and is based on K\"uneth's formula, the Frobenius map on $Gal (\bar k / k)$, Chebotarev's theorem, and Faltings theorem.  

We will generalize such theorems for curves $C_{r, \l, m}$ over an algebraically closed field $k$ of characteristic relatively prime to $r$. Our proof is based solely on automorphisms of curves and it is much simpler than in \cite{Ya}.

Let  $\X_{r, s}$ be a generic algebraic  curve  defined over an algebraically closed field  $k$   and $C_{r, \l, m}$ as above. Then we have the following.

\begin{thm}
The Jacobian $\Jac  (\X_{r, s})$ is isogenous to the product of the $C_{r, \l, m}$, for $1 \leq \l \leq s$, namely 
\[  \Jac (\X_{r, s}) \iso  \prod_{1 = \l}^s   \Jac (C_{r, \l, m}),   \] 
if and only if 
\begin{equation} r=4\cdot\frac{1+s-2^s}{ms(s+1)-s\cdot2^{s+1}}\label{num-thy-eqn}.
\end{equation}
\end{thm}

\proof
We denote by $\s_i (x, y_i, z) \to (x, \e_r y_i, z)$, for $i=1, \dots , s$.  Then the quotient spaces $\X_{r, s} / \< \s_i \>$ are  the curves $C_{r, i, s}$, for  $i=1, \dots , s$.  Since $\s_i$ is a central element in $G = \Aut (\X_{r, s} )$ then $H_i := \< \s_i \> \normal G$, for all $i=1, \dots , s$.  Obviously, for all $i \neq j$ we have $H_i \cap H_j = \{ e \}$.  Hence, $H_1, \dots , H_s$ forms a partition for $G$.

The genus for every $C_{r, i, s}$, by Lemma~\ref{lem_1} is given by Eq.~\eqref{eq_g}.   Then we have 
\[
\begin{split}
\sum_{\l=1}^s \, g \left( C_{r, \l, m} \right)  & = \sum_{\l=1}^s \left( 1 + \frac r 2 \left((r-1)\l m -2  \right)  \right)  \\
& = s(r-1) \left( \frac r 4 m (s+1) -1    \right).   \\
\end{split}
 \]
From the results of Eq.~\eqref{dec_2}  we have that 
\[  \frac r 4  m s (s+1) - s   =  rs \cdot 2^{s-1} - 2^s +1.  \]
Hence, 
\[ r = 4 \cdot \frac {1+s-2^s} {ms(s+1)- s \cdot 2^{s+1}}.\]
This completes the proof.
\qed

\begin{rem}
For $m=2$ this result is the case of Theorem~4.2 in \cite{Ya}. We get  $r= \frac  2 s$. Hence, $s=1$ or $s=2$. 
Therefore,    Theorem~4.2 in \cite{Ya} is true only for curves $F_{m, 1}$ or $F_{m, 2}$. 
\end{rem}

Next we determine integer combinations of $r,m$, and $s$ that satisfy   Eq.~\eqref{num-thy-eqn}.  First we need the following lemma.

\begin{lemma}
For integers $a, n$ with $n>1$, suppose $a^n\equiv1\text{ (mod }n)$.  Let $p$ be the smallest prime divisor of $n$.  Then $a\equiv1\text{ (mod }p)$.
\end{lemma} 

\begin{proof}
Suppose $a^n\equiv1\text{ (mod }n)$ for integers $a,n$ with $n>1$.  Let $p$ be the smallest prime divisor of $n$.  Since $a^n\equiv1\text{ (mod }n)$, then $a^n\equiv1\text{ (mod }p)$.  Let $d$ be the order of $a$ modulo $p$.  Since $a^{p-1}\equiv1\text{ (mod }p)$, $d$ divides $p-1$, so $d<p$.  And since $a^{n}\equiv1\text{ (mod }p)$, $d$ divides $n$.  However, since $p$ is the smallest prime divisor of $n$, the only divisor of $n$ which is less than $p$ is 1.  Hence, $d=1$, so the order of $a$ modulo $p$ is 1, so $a\equiv1\text{ (mod }p)$, as desired.

\end{proof}


\begin{prop}Suppose $r,m,s\in\mathbb{N}$ satisfy Eq.~\ref{num-thy-eqn}.  Then $mrs=4k$ for some odd integer $k$. Moreover, 

i) If $s\equiv1\text{ (mod 2)}$, then $s=1$.

ii) If $s\equiv2\text{ (mod 4)}$, then $s=2t$ for some odd integer $t$ which satisfies $4^t\equiv1\text{ (mod }t)$.  Furthermore, $t$ is a multiple of 3.

iii) If $s\equiv0\text{ (mod 4)}$, then $s=4u$ for some odd integer $u$ which satisfies $16^u\equiv1\text{ (mod }u)$.  Furthermore, $u$ is a multiple of 3 or 5.
\end{prop}


\begin{proof}
We first determine possible values of $s$ for which $r$, which is given by the equation \[ r = 4 \cdot \frac {1+s-2^s} {ms(s+1)- s \cdot 2^{s+1}},\] can be an integer.  We proceed in cases, considering the highest power of 2 that divides $s$.  In particular, the powers that we consider are 0, 1, 2, and then at least 3.  


First, suppose $s$ is an odd integer.  Since $s$ divides the denominator, it follows that $s$ divides $4(1+s-2^s)$.  Since $\gcd(s,4)=1$, we conclude that $2^s\equiv 1\text{ (mod $s$)}$.  If $s=1$, then this congruence is satisfied.  Now, suppose $s>1$ and let $p$ be the smallest prime divisor of $s$.  By the Lemma above, one has $2\equiv 1\text{ (mod $p$)}$, so $p$ divides 1, which is impossible.  Thus, if $s$ is odd, then $s=1$.

Also, note that in this case if $s=1$, then one finds that $m=2$ and $r=2$.

Next, suppose $s=2t$ for some odd integer $t$.  Substituting in, we get \[ r = 4 \cdot \frac {1+2t-2^{2t}} {m(2t)(2t+1)- 2t \cdot 2^{2t+1}}=2\cdot\frac{1+2t-4^t}{mt(2t+1)-2t\cdot4^{t}}.\]  As above, since $t$ divides the denominator, $t$ divides $2(1+2t-4^t)$.  Since $\gcd(t,2)=1$, we conclude that $4^t\equiv 1\text{ (mod $t$)}$.  If $t=1$, then this congruence is satisfied.  Now, suppose $t>1$ and let $p$ be the smallest prime divisor of $t$.  By the Lemma above, one has $4\equiv1\text{ (mod $p$)}$, which implies $p$ divides 3, so $p=3$.  Thus, $s=2t$ for some integer $t$ which is either $1$ or is a multiple of 3 and which satisfies $4^t\equiv1\text{ (mod $t$)}$.  

Next, suppose $s=4u$ for some odd integer $u$.  Substituting in, we get \[ r = 4 \cdot \frac {1+4u-2^{4u}} {m(4u)(4u+1)- 4u \cdot 2^{4u+1}}=\frac{1+4u-16^u}{mu(4u+1)-2u\cdot16^{u}}.\]  As above, since $u$ divides the denominator, $u$ divides $(1+4u-16^u)$.  Thus, $16^u\equiv 1\text{ (mod $u$)}$.  If $u=1$, then this congruence is satisfied.  Now, suppose $u>1$ and let $p$ be the smallest prime divisor of $u$.  By the Lemma above, one has $16\equiv1\text{ (mod $p$)}$, which implies $p$ divides 15, so $p=3$ or $p=5$.  Thus, $s=4u$ for some integer $u$ which is either $1$ or is a multiple of either 3 or 5 and which satisfies $16^u\equiv1\text{ (mod $u$)}$.  

Finally suppose $s=8v$ for some (even or odd) integer $v$.  Substituting in, we get \[ r = 4 \cdot \frac {1+8v-2^{8v}} {m(8v)(8v+1)- 8v \cdot 2^{8v+1}}=\frac{1+8v-2^{8v}}{2mv(8v+1)-2v\cdot2^{8v+1}}.\]  Thus, since $2$ divides the denominator, $2$ divides $(1+8v-2^{8v})$, so $2^{8v}\equiv 1\text{ (mod 2)}$, which occurs only if $v=0$.  Thus, $s=0$.

To show $mrs=4k$ for some odd integer $k$, we have two cases to consider.  If $s$ is odd, then $s=1$, so $m=r=2$ and thus $mrs=4$.  If $s$ is even, then we clear denominators of Eq.~\ref{num-thy-eqn} and consider the equation modulo 8 to get 
\[rms(s+1)-rs2^{s+1}\equiv 4\cdot(1+s-2^s)\text{ (mod $8$)}.\]  
Since $s$ is even and thus at least $2$, this simplifies to $rms\equiv4\text{ (mod $8$)}$, so $mrs=4k$ for some odd integer $k$.
\end{proof}

\begin{rem}
i) When $s=2t$ as in the proof of the theorem, then the first few integer values of $t$ are $1, 3, 9, 21, 27, 63, 81, 147, 171, 189, 243$.  This is sequence A014945 in the Online Encyclopedia of Integer Sequences.

ii) When $s=4u$ as in the proof of the theorem, then  the first few integer values of $u$ are $1, 3, 5, 9, 15, 21, 25, 27, 39, 45, 55, 63, 75, 81, 105, 117$.  This is sequence A014957 in the Online Encyclopedia of Integer Sequences.
\end{rem}

We then search for values of $s$ that satisfy the above proposition.  Searching among $1\leq s < 500$, we find the following possible values for $s$.  \begin{align*}\{1, 2, 4, 6, 12, 18, 20, 36, 42, 54, 60, 84, 100, 108, 126, 156, 162, \\ 180, 220, 252, 294, 300, 324, 342, 378, 420, 468, 486, \dots \}.\end{align*}  We now check each of these to find corresponding integer values of $m$ and $r$.

\begin{cor}
Suppose $r,m,s\in\mathbb{N}$ satisfy Eq.~\ref{num-thy-eqn}, and suppose $1\leq s<500$ with $s\neq 300, 420, 468$.  Then $s\in\{1,2,6,18,42,126,162,294,378,486\}.$  We display the combinations of integers $s,m$, and $r$ in the table below.  Note that $m$ and $r$ grow quickly relative to $s$, hence the scientific notation for the cases where $s>100$.

\begin{center}
\begin{tabular}{|c|c|c|}\hline
$s$ & $m$ & $r$ \\ \hline \hline
$1$ & $2$ & $2$ \\ \hline
$2$ & $2$ & $1$ \\ \hline
$6$ & $18$ & $19$ \\ \hline
$18$ & $27594$ & $29125$ \\ \hline
$42$ & $204560302842$ & $209430786241$ \\ \hline
$126$ & $1.3397\times10^{36}$ & $1.3503\times10^{36}$ \\ \hline
$162$ & $7.1730\times10^{46}$ & $7.2173\times10^{46}$ \\ \hline
$294$ & $2.1579\times10^{86}$ & $2.8391\times10^{62}$ \\ \hline
$294$ & $2.1579\times10^{86}$ & $3.3025\times10^{31}$ \\ \hline
$294$ & $2.1579\times10^{86}$ & $2.2665\times10^{27}$ \\ \hline
$378$ & $3.2488\times10^{111}$ & $3.2574\times10^{111}$ \\ \hline
$486$ & $8.2050\times10^{143}$ & $8.2219\times10^{143}$ \\ \hline
$486$ & $8.2050\times10^{143}$ & $1.4596\times10^{20}$ \\ \hline
\end{tabular}
\end{center}
\end{cor}

Note that there are multiple combinations of integers $r$ and $m$ when $s=294$ or $486$.  Also note that we were unable to get results for $s=300, 420$, or $468$.  This is because of the time required to factor $(1+s-2^s)$ in those cases.


Interestingly, we do not have any cases where $s=4u$ for some odd integer $u$.


\medskip

\noindent \textbf{Acknowledgments:} The authors want to thank the anonymous referee for helpful comments and remarks. 


\end{document}